\documentclass{amsart}
\usepackage{amsmath}
\usepackage{amsthm}
\usepackage{amssymb}

\usepackage{tikz}
\usetikzlibrary{arrows}

\usepackage{clontzDefinitions}

      \theoremstyle{plain}
      \newtheorem{theorem}{Theorem}
      
      \newtheorem{corollary}[theorem]{Corollary}
      \newtheorem{proposition}[theorem]{Proposition}
      
      \newtheorem{question}[theorem]{Question}

      \theoremstyle{definition}
      \newtheorem{definition}[theorem]{Definition}

      \theoremstyle{remark}

      \theoremstyle{plain}
      \newtheorem*{theorem*}{Theorem}
      \newtheorem*{lemma*}{Lemma}
      \newtheorem*{corollary*}{Corollary}
      \newtheorem*{proposition*}{Proposition}
      \newtheorem*{conjecture*}{Conjecture}
      \newtheorem*{question*}{Question}
      \newtheorem*{claim*}{Claim}

      \theoremstyle{definition}
      \newtheorem*{definition*}{Definition}
      \newtheorem*{example*}{Example}
      \newtheorem*{game*}{Game}

      \theoremstyle{remark}
      \newtheorem*{remark*}{Remark}

\begin{document}

\title{Dual selection games}

\author{Steven Clontz}
\address{Department of Mathematics and Statistics,
The University of South Alabama,
Mobile, AL 36688}
\email{sclontz@southalabama.edu}



\begin{abstract}
  Often, a given selection game studied in the literature has
  a known dual game. In dual games, a winning
  strategy for a player in either game may be used to create
  a winning strategy for the opponent in the dual. 
  For example, the Rothberger selection game involving open covers
  is dual to the point-open game. This extends to a general
  theorem: if \(\{\ran{f}:f\in\mathbf C(\mc R)\}\) is coinitial in \(\mc A\)
  with respect to \(\subseteq\),
  where \(\mathbf C(\mc R)=\{f\in(\bigcup\mc R)^{\mc R}:R\in\mc R\Rightarrow f(R)\in R\}\) 
  collects the choice functions on the set \(\mc R\),
  then \(G_1(\mc A,\mc B)\) and \(G_1(\mc R,\neg\mc B)\)
  are dual selection games. 
\end{abstract}

\maketitle

\section{Introduction}

\begin{definition}
  An \term{\(\omega\)-length game} is a pair \(G=\<M,W\>\) such that
  \(W\subseteq M^\omega\). The set \(M\) is the \term{moveset} of the game,
  and the set \(W\) is the \term{payoff set} for the second player.
\end{definition}

In such a game \(G\), players \(\plI\) and \(\plII\) alternate making choices
\(a_n\in M\) and \(b_n\in M\) during each round \(n<\omega\), 
and \(\plII\) wins the game if and only if \(\<a_0,b_0,a_1,b_1,\dots\>\in W\).

Often when defining games, \(\plI\) and \(\plII\) are restricted to choosing
from different movesets \(A,B\). Of course, this can be modeled with \(\<M,W\>\)
by simply letting \(M=A\cup B\) and adding/removing sequences from \(W\)
whenever player \(\plI\)/\(\plII\) makes the first ``illegal'' move.

A class of such games heavily studied in the literature (see \cite{MR1378387} and its
many sequels) are selection games.

\begin{definition}
  The \term{selection game} \(\schStrongSelGame{\mc A}{\mc B}\) 
  is an \(\omega\)-length game involving Players \(\plI\) and \(\plII\). 
  During round \(n\), \(\plI\) chooses
  \(A_n\in\mc A\), followed by \(\plII\) choosing \(B_n\in A_n\).
  Player \(\plII\) wins in the case that \(\{B_n:n<\omega\}\in\mc B\),
  and Player \(\plI\) wins otherwise.
\end{definition}

  For brevity, let 
  \[
    \schStrongSelGame{\mc A}{\neg \mc B}
      =
    \schStrongSelGame{\mc A}{\mc P\left(\bigcup \mc A\right)\setminus \mc B}
  .\]
  That is, \(\plII\) wins in the case that \(\{B_n:n<\omega\}\not\in\mc B\),
  and \(\plI\) wins otherwise.

\begin{definition}
  For a set \(X\), let \(\mathbf C(X)=\{f\in(\bigcup X)^X:x\in X\Rightarrow f(x)\in x\}\)
  be the collection of all choice functions on \(X\).
\end{definition}

\begin{definition}
  Write \(X\preceq Y\) if \(X\) is coinitial in \(Y\) with respect to \(\subseteq\);
  that is, \(X\subseteq Y\), and for all \(y\in Y\), there exists \(x\in X\) such that 
  \(x\subseteq y\).

  In the context of selection games, we will say \(\mc A'\) is a \term{selection basis}
  for \(\mc A\) when \(\mc A'\preceq \mc A\).
\end{definition}

\begin{definition}
  The set \(\mc R\) is said to be a \term{reflection} of the set \(\mc A\)
  if \[\{\ran f:f\in\mathbf C(\mc R)\}\] is a selection basis for \(\mc A\).
\end{definition}

Put another way, \(\mc R\) is a reflection of \(\mc A\) if for every \(A\in\mc A\),
there exists \(f\in\mathbf C(\mc R)\) such that \(\ran f\in\mc A\) and \(\ran f\subseteq A\).

As we will see, reflections of selection sets are used frequently (but implicitly) 
throughout the literature to define dual selection games.

We use the following conventions to describe strategies for playing games.

\begin{definition}
  For \(f\in B^A\) and \(X\subseteq A\), let \(f\rest X\) be the restrction of \(f\)
  to \(X\). In particular, for \(f\in B^\omega\) and \(n<\omega\), \(f\rest n\)
  describes the first \(n\) terms of the sequence \(f\).
\end{definition}

\begin{definition}
  A \term{strategy} for the first player \(\plI\) (resp. second player \(\plII\))
  in a game \(G\) with moveset \(M\) is a function
  \(\sigma:M^{<\omega}\to M\). This strategy is said to be \term{winning} if
  for all possible \term{attacks} \(\alpha\in M^\omega\) by their opponent,
  where \(\alpha(n)\) is played by the opponent during round \(n\),
  the player wins the game by playing \(\sigma(\alpha\rest n)\)
  (resp. \(\sigma(\alpha\rest n+1)\)) during round \(n\).
\end{definition}

That is, a strategy is a rule that determines the moves of a player based upon
all previous moves of the opponent. (It could also rely on all previous
moves of the player using the strategy, since these can be reconstructed from
the previous moves of the opponent and the strategy itself.)

\begin{definition}
  A \term{predetermined strategy} for the first player \(\plI\)
  in a game \(G\) with moveset \(M\) is a function
  \(\sigma:\omega\to M\). This strategy is said to be winning if
  for all possible attacks \(\alpha\in M^\omega\) by their opponent,
  the first player wins the game by playing \(\sigma(n)\)
  during round \(n\).
\end{definition}

So a predetermined strategy ignores all moves of the opponent during the
game (all moves were decided before the game began).

\begin{definition}
  A \term{Markov strategy} for the second player \(\plII\)
  in a game \(G\) with moveset \(M\) is a function
  \(\sigma:M\times\omega\to M\). This strategy is said to be winning if
  for all possible attacks \(\alpha\in M^\omega\) by their opponent,
  the first player wins the game by playing \(\sigma(\alpha(n),n)\)
  during round \(n\).
\end{definition}

So a Markov strategy may only consider the most recent move of the opponent,
and the current round number. Note that unlike perfect-information or
predetermined strategies, a Markov strategy cannot use knowledge of
moves used previously by the player (since they depend on previous moves of
the opponent that have been ``forgotten'').

\begin{definition}
  Write \(\plI\win G\) (resp. \(\plI\prewin G\)) if player \(\plI\) has a winning
  strategy (resp. winning predetermined strategy) for the game \(G\). Similarly,
  write \(\plII\win G\) (resp. \(\plII\markwin G\)) if player \(\plII\) has a winning
  strategy (resp. winning Markov strategy) for the game \(G\).
\end{definition}

Of course, \(\plII\markwin G\Rightarrow \plII\win G\Rightarrow \plI\notwin G\Rightarrow \plI\notprewin G\).
In general, none of these implications (not even the second \cite{MR0054922}) can be reversed.

It's worth noting that \(\plI\notprewin \schStrongSelGame{\mc A}{\mc B}\) is equivalent
to the selection principle often denoted \(\schStrongSelProp{\mc A}{\mc B}\) in the literature.

The goal of this paper is to characerize when two games are ``dual'' in the following
senses.

\begin{definition}
  A pair of games \(G(X),H(X)\) defined for a topological space \(X\)
  are \term{Markov information dual} if both
  of the following hold.
  \begin{itemize}
    \item \(I\prewin G(X)\) if and only if \(II\markwin H(X)\).
    \item \(II\markwin G(X)\) if and only if \(I\prewin H(X)\).
  \end{itemize}
\end{definition}

\begin{definition}
  A pair of games \(G(X),H(X)\) defeind for a topological space \(X\)
  are \term{perfect information dual} if both
  of the following hold.
  \begin{itemize}
    \item \(I\win G(X)\) if and only if \(II\win H(X)\).
    \item \(II\win G(X)\) if and only if \(I\win H(X)\).
  \end{itemize}
\end{definition}

\section{Main Results}

The following four theorems demonstrate that reflections characterize
dual selection games for both perfect information strategies and
certain limited information strategies.

The duality of the Rothberger game \(\schStrongSelGame{\mc O_X}{\mc O_X}\)
and the point-open game on \(X\) for perfect information strategies
was first noted by Galvin in \cite{MR0493925}, and
for Markov-information strategies by Clontz and Holshouser in
\cite{2018arXiv180606001C}. These proofs may be generalized as follows.

\begin{theorem}
  Let \(\mc R\) be a reflection of \(\mc A\). 

  Then
  \(\plI\prewin\schStrongSelGame{\mc A}{\mc B}\) if and only if
  \(\plII\markwin\schStrongSelGame{\mc R}{\neg\mc B}\).
\end{theorem}

\begin{proof}
  Let \(\sigma\) witness 
  \(\plI\prewin\schStrongSelGame{\mc A}{\mc B}\).
  Since \(\sigma(n)\in\mc A\),
  \(\ran{f_n}\subseteq\sigma(n)\)
  for some \(f_n\in\mathbf C(\mc R)\). So let
  \(\tau(R,n)=f_n(R)\) for all \(R\in \mc R\) and \(n<\omega\).
  Suppose \(R_n\in \mc R\) for all \(n<\omega\).
  Note that since \(\sigma\) is winning and 
  \(\tau(R_n,n)=f_n(R_n)\in\ran{f_n}\subseteq\sigma(n)\),
  \(\{\tau(R_n,n):n<\omega\}\not\in\mc B\). Thus \(\tau\) witnesses
  \(\plII\markwin\schStrongSelGame{\mc R}{\neg\mc B}\).

  Now let \(\sigma\) witness
  \(\plII\markwin\schStrongSelGame{\mc R}{\neg\mc B}\).
  Let \(f_n\in\mathbf C(\mc R)\) be defined by \(f_n(R)=\sigma(R,n)\),
  and let \(\tau(n)=\ran{f_n}\in\mc A\). 
  Suppose that \(B_n\in\tau(n)=\ran{f_n}\) for
  all \(n<\omega\). Choose \(R_n\in\mc R\) such that 
  \(B_n=f_n(R_n)=\sigma(R_n,n)\). Since \(\sigma\) is winning,
  \(\{B_n:n<\omega\}\not\in\mc B\). Thus \(\tau\) witnesses
  \(\plI\prewin\schStrongSelGame{\mc A}{\mc B}\).
\end{proof}

\begin{theorem}
  Let \(\mc R\) be a reflection of \(\mc A\). 

  Then
  \(\plII\markwin\schStrongSelGame{\mc A}{\mc B}\) if and only if
  \(\plI\prewin\schStrongSelGame{\mc R}{\neg\mc B}\).
\end{theorem}

\begin{proof}
  Let \(\sigma\) witness 
  \(\plII\markwin\schStrongSelGame{\mc A}{\mc B}\).
  Let \(n<\omega\). Suppose that for each \(R\in\mc R\),
  there was \(g(R)\in R\) such that for all \(A\in \mc A\),
  \(\sigma(A,n)\not=g(R)\). Then \(g\in\mathbf C(\mc R)\)
  and \(\ran g\in\mc A\),
  thus \(\sigma(\ran g,n)\not=g(R)\) for all \(R\in\mc R\),
  a contradiction.

  So choose \(\tau(n)\in\mc R\) such that for all \(r\in \tau(n)\)
  there exists \(A_{r,n}\in\mc A\) such that \(\sigma(A_{r,n},n)=r\).
  It follows that when \(r_n\in\tau(n)\) for \(n<\omega\),
  \(\{r_n:n<\omega\}=\{\sigma(A_{r_n,n}):n<\omega\}\in B\),
  so \(\tau\) witnesses
  \(\plI\prewin\schStrongSelGame{\mc R}{\neg\mc B}\).

  Now let \(\sigma\) witness 
  \(\plI\prewin\schStrongSelGame{\mc R}{\neg\mc B}\).
  Then \(\sigma(n)\in\mc R\), so for \(A\in\mc A\), let
  \(f_A\in\mathbf C(\mc R)\) satisfy \(\ran{f_A}\subseteq A\),
  and let \(\tau(A,n)=f_A(\sigma(n))\in A\cap\sigma(n)\).
  Then if \(A_n\in\mc A\) for \(n<\omega\), \(\tau(A_n,n)\in\sigma(n)\),
  so \(\{\tau(A_n,n):n<\omega\}\in\mc B\).
  Thus \(\tau\) witnesses
  \(\plII\markwin\schStrongSelGame{\mc A}{\mc B}\).
\end{proof}

\begin{theorem}
  Let \(\mc R\) be a reflection of \(\mc A\). 

  Then
  \(\plI\win\schStrongSelGame{\mc A}{\mc B}\) if and only if
  \(\plII\win\schStrongSelGame{\mc R}{\neg\mc B}\).
\end{theorem}

\begin{proof}
  Let \(\sigma\) witness 
  \(\plI\win\schStrongSelGame{\mc A}{\mc B}\).
  Let \(c(\emptyset)=\emptyset\). Suppose 
  \(c(s)\in(\bigcup A)^{<\omega}\)
  is defined for \(s\in\mc R^{<\omega}\). Since \(\sigma(c(s))\in\mc A\),
  let \(f_s\in\mathbf C(\mc R)\) satisfy \(\ran{f_s}\subseteq\sigma(c(s))\),
  and let \(c(s\concat\<R\>)=c(s)\concat\<f_s(R)\>\).
  Then let \(c(\alpha)=\bigcup\{c(\alpha\rest n):n<\omega\}\)
  for \(\alpha\in\mc R^\omega\), so
  \[
    c(\alpha)(n)
      =
    f_{\alpha\rest n}(\alpha(n))
      \in
    \ran{f_{\alpha\rest n}}
      \subseteq
    \sigma(c(\alpha\rest n))
  \]
  demonstrating that \(c(\alpha)\) is a legal attack against \(\sigma\).

  Let \(\tau(s\concat\<R\>)=f_s(R)\). Consider the attack \(\alpha\in\mc R^\omega\)
  against \(\tau\). Then since \(\sigma\) is winning and
  \(
    \tau(\alpha\rest n+1)=f_{\alpha\rest n}(\alpha(n))\in
    \ran{f_{\alpha\rest n}}\subseteq\sigma(c(\alpha\rest n))
  \), it follows that \(\{\tau(\alpha\rest n+1):n<\omega\}\not\in\mc B\).
  Thus \(\tau\) witnesses
  \(\plII\win\schStrongSelGame{\mc R}{\neg\mc B}\).

  Now let \(\sigma\) witness
  \(\plII\win\schStrongSelGame{\mc R}{\neg\mc B}\).
  For \(s\in \mc R^{<\omega}\), define \(f_s\in\mathbf C(\mc R)\)
  by \(f_s(R)=\sigma(s\concat\<R\>)\). Let \(\tau(\emptyset)=\ran{f_\emptyset}\in\mc A\),
  and for \(x\in\tau(\emptyset)\), choose \(R_{\<x\>}\in\mc R\) such that
  \(x=f_{\emptyset}(R_{\<x\>})\) (for other \(x\in\bigcup A\), choose \(R_{\<x\>}\)
  arbitrarily as it won't be used). Now let \(s\in(\bigcup A)^{<\omega}\),
  and suppose \(R_{s\rest n\concat\<x\>}\in\mc R\) has been defined for
  \(n\leq|s|\) and \(x\in\bigcup A\). 
  Then let \(\tau(s\concat\<x\>)=\ran{f_{\<R_{s\rest 0},\dots,R_s,R_{s\concat\<x\>}\>}}\)
  and for \(y\in\tau(s)\) choose \(R_{s\concat\<x,y\>}\) such that
  \(x=f_{\<R_{s\rest 0},\dots,R_s,R_{s\concat\<x\>}\>}(R_{s\concat\<x,y\>})\) (and again,
  choose \(R_{s\concat\<x,y\>}\) arbitrarily for other \(y\in\bigcup\mc A\) as it won't be used).

  Then let \(\alpha\) attack \(\tau\), so
  \(\alpha(n)\in\tau(\alpha\rest n)\) and thus 
  \(\alpha(n)=f_{\<R_{\alpha\rest 0},\dots,R_{\alpha\rest n}\>}(R_{\alpha\rest n+1})
  =\sigma(\<R_{\alpha\rest 0},\dots,R_{\alpha\rest n+1}\>)\).  Since \(\sigma\) is winning,
  \(\{\sigma(\<R_{\alpha\rest 0},\dots,R_{\alpha\rest n+1}\>):n<\omega\}
  =\{\alpha(n):n<\omega\}\not\in\mc B\).
  Thus \(\tau\) witnesses
  \(\plI\win\schStrongSelGame{\mc A}{\mc B}\).
\end{proof}

\begin{theorem}
  Let \(\mc R\) be a reflection of \(\mc A\). 

  Then
  \(\plII\win\schStrongSelGame{\mc A}{\mc B}\) if and only if
  \(\plI\win\schStrongSelGame{\mc R}{\neg\mc B}\).
\end{theorem}

\begin{proof}
  Let \(\sigma\) witness 
  \(\plII\win\schStrongSelGame{\mc A}{\mc B}\).
  Let \(s\in(\bigcup A)^{<\omega}\) and assume \(a(s)\in\mc A^{|s|}\) is defined
  (of course, \(a(\emptyset)=\emptyset\)).
  Suppose for all \(R\in\mc R\) there existed \(f(R)\in R\) such that for all
  \(A\in\mc A\), \(\sigma(a(s)\concat\<A\>)\not=f(R)\). Then
  \(f\in\mathbf C(\mc R)\) and \(\ran{f}\in\mc A\), and thus 
  \(\sigma(a(s)\concat\<\ran{f}\>)\not=f(R)\) for all \(R\in\mc R\), a contradiction.
  So let \(\tau(s)\in\mc R\) satisfy for all \(x\in\tau(s)\) there exists
  \(a(s\concat\<x\>)\in\mc A^{|s|+1}\) extending \(a(s)\) such that 
  \(x=\sigma(a(s\concat\<x\>))\).

  If \(\tau\) is attacked by \(\alpha\in(\bigcup R)^\omega\), then 
  \(\alpha(n)\in\tau(\alpha\rest n)\). So \(\alpha(n)=\sigma(a(\alpha\rest n+1))\),
  and since \(\sigma\) is winning, 
  \(\{\sigma(a(\alpha\rest n+1)):n<\omega\}=\{\alpha(n):n<\omega\}\in\mc B\).
  Therefore \(\tau\) witnesses
  \(\plI\win\schStrongSelGame{\mc R}{\neg\mc B}\).

  Now let \(\sigma\) witness
  \(\plI\win\schStrongSelGame{\mc R}{\neg\mc B}\).
  Let \(s\in\mc A^{<\omega}\), and suppose \(r(s)\in(\bigcup \mc R)^{|s|}\) is defined
  (again, \(r(\emptyset)=\emptyset\)). For \(A\in\mc A\) choose \(f_A\in\mathbf C(\mc R)\)
  where \(\ran{f_A}\subseteq A\), and let \(\tau(s\concat\<A\>)=f_A(\sigma(r(s)))\),
  and let \(r(s\concat\<A\>)\) extend \(r(s)\) by letting
  \(r(s\concat\<A\>)(|s|)=\tau(s\concat\<A\>)\).

  If \(\tau\) is attacked by \(\alpha\in\mc A^\omega\), 
  then since \(\tau(\alpha\rest n+1)=f_{\alpha(n)}(\sigma(r(\alpha\rest n))\in\alpha(n)\cap\sigma(r(\alpha\rest n))\)
  and \(\sigma\) is winning, we conclude that \(\tau\) is a legal strategy and
  \(\{\tau(\alpha\rest n+1):n<\omega\}\in\mc B\).
  Therefore \(\tau\) witnesses
  \(\plII\win\schStrongSelGame{\mc A}{\mc B}\).
\end{proof}

\begin{corollary}
  If \(\mc R\) is a reflection of \(\mc A\),
  then \(\schStrongSelGame{\mc A}{\mc B}\) and \(\schStrongSelGame{\mc R}{\neg\mc B}\)
  are both perfect information dual and Markov information dual.
\end{corollary}

\section{Applications of Reflections}

\begin{definition}\label{selectionSets}
  Let \(X\) be a topological space and \(\mc T_X\) be a chosen basis of nonempty sets for its topology.
  \begin{itemize}
    \item Let \(\mc T_{X,x} = \{U\in\mc T_X : x\in U\}\) be the local point-base at \(x\in X\).
    \item Let \(\Omega_{X,x} = \{Y\subseteq X: \forall U\in\mc T_{X,x}(U\cap Y\not=\emptyset)\}\) be the fan at \(x\in X\).
    \item Let \(\mc T_{X,F} = \{U\in\mc T_X : F\subseteq U\}\) be the local finite-base at \(F\in[X]^{<\aleph_0}\).
    \item Let \(\mc O_X = \{\mc U\subseteq\mc T_X : \bigcup \mc U=X\}\) be the collection
          of basic open covers of \(X\).
    \item Let \(\mc P_X = \{\mc T_{X,x} : x\in X\}\) be the collection of local point-bases of \(X\).
    \item Let \(\Omega_X = \{\mc U\subseteq\mc T_X : \forall F\in[X]^{<\aleph_0}\exists U\in\mc U(F\subseteq U)\}\)
          be the collection of basic \(\omega\)-covers of \(X\).
    \item Let \(\mc F_X = \{\mc T_{X,F} : F\in [X]^{<\aleph_0}\}\) be the collection of local finite-bases of \(X\).
    \item Let \(\mc D_X = \{Y\subseteq X: \forall U\in\mc T_X(U\cap Y\not=\emptyset)\}\) be the collection of dense subsets of \(X\).
    \item Let \(\Gamma_{X,x} = \{Y\subseteq X: \forall U\in\mc T_{X,x}(Y\setminus U\in[X]^{<\aleph_0})\}\) be the collection
          of converging fans at \(x\in X\). (When intersected with \([X]^{\aleph_0}\), these are the non-trivial
          sequences of \(X\) converging to \(x\).)
  \end{itemize}
\end{definition}

While these notions were defined in terms of a particular basis, the reader may verify the the following.

\begin{proposition}
  Let \(\mc A'\) be a selection basis for \(\mc A\).
  \begin{itemize}
    \item \(\plI\win\schStrongSelGame{\mc A}{\mc B}\Leftrightarrow\plI\win\schStrongSelGame{\mc A'}{\mc B}\).
    \item \(\plI\prewin\schStrongSelGame{\mc A}{\mc B}\Leftrightarrow\plI\prewin\schStrongSelGame{\mc A'}{\mc B}\).
    \item \(\plII\win\schStrongSelGame{\mc A}{\mc B}\Leftrightarrow\plII\win\schStrongSelGame{\mc A'}{\mc B}\).
    \item \(\plII\markwin\schStrongSelGame{\mc A}{\mc B}\Leftrightarrow\plII\markwin\schStrongSelGame{\mc A'}{\mc B}\).
  \end{itemize}
\end{proposition}

\begin{proposition}
  Each selection set in Definition \ref{selectionSets} is a selection basis for the set
  defined by replacing \(\mc T_X\) with the set of all nonempty open sets in \(X\).
\end{proposition}

As such, the choice of topological basis is irrelevant when playing selection games using these sets.

We may now establish (or re-establish) the following dual games.

\begin{proposition}
  \(\mc P_X\) is a reflection of \(\mc O_X\).
\end{proposition}
\begin{proof}
  For every open cover \(\mc U\), the corresponding choice function \(f\in\mathbf C(\mc P_X)\) is simply
  the witness that \(x\in f(\mc T_{X,x})\in\mc U\).
\end{proof}

\begin{corollary}
  \(\schStrongSelGame{\mc O_X}{\mc B}\) and \(\schStrongSelGame{\mc P_X}{\neg\mc B}\) are perfect-information
  and Markov-information dual.
\end{corollary}

In the case that \(\mc B=\mc O_X\), \(\schStrongSelGame{\mc O_X}{\mc O_X}\) is the well-known Rothberger game,
and \(\schStrongSelGame{\mc P_X}{\neg\mc O_X}\) is isomorphic
to the point-open game \(PO(X)\): \(\plI\) chooses points of \(X\), \(\plII\) chooses an open neighborhood
of each chosen point, and \(\plI\) wins if \(\plII\)'s choices are a cover.
So this was simply the classic result that the Rothberger game and
point-open game are perfect-information dual \cite{MR0493925}, and the more recent result that
these games are Markov-information dual \cite{2018arXiv180606001C}.

\begin{proposition}
  \(\mc F_X\) is a reflection of \(\Omega_X\).
\end{proposition}
\begin{proof}
  For every \(\omega\)-cover \(\mc U\), the corresponding choice function \(f\in\mathbf C(\mc F_X)\) is simply
  the witness that \(F\subseteq f(\mc T_{X,F})\in\mc U\).
\end{proof}

\begin{corollary}
  \(\schStrongSelGame{\Omega_X}{\mc B}\) and \(\schStrongSelGame{\mc F_X}{\neg\mc B}\) are perfect-information
  and Markov-information dual.
\end{corollary}

Note that in the case that \(\mc B=\Omega_X\), \(\schStrongSelGame{\Omega_X}{\Omega_X}\) is the Rothberger
game played with \(\omega\)-covers, and \(\schStrongSelGame{\mc F_X}{\neg\Omega_X}\) is isomorphic
to the \(\Omega\)-finite-open game \(\Omega FO(X)\): \(\plI\) chooses finite subsets of \(X\), \(\plII\) chooses an open neighborhood
of each chosen finite set, and \(\plI\) wins if \(\plII\)'s choices are an \(\omega\)-cover.
These games were shown to be dual in \cite{2018arXiv180606001C}.

\begin{proposition}
  \(\mc T_X\) is a reflection of \(\mc D_X\).
\end{proposition}
\begin{proof}
  For every dense \(D\), the corresponding choice function \(f\in\mathbf C(\mc T_X)\) is simply
  the witness that \(f(U)\in U\cap D\).
\end{proof}

\begin{corollary}
  \(\schStrongSelGame{\mc D_X}{\mc B}\) and \(\schStrongSelGame{\mc T_X}{\neg\mc B}\) are perfect-information
  and Markov-information dual.
\end{corollary}

In the case that \(\mc B=\Omega_{X,x}\) for some \(x\in X\), \(\schStrongSelGame{\mc D_X}{\Omega_{X,x}}\) is
the strong countable dense fan-tightness game at \(x\), see e.g. \cite{MR2678950}. \(\schStrongSelGame{\mc T_X}{\neg\Omega_{X,x}}\)
is the game \(CL(X,x)\) first studied by Tkachuk in \cite{tkachukTwoPointGame}. Tkachuk showed in that paper
that these games are perfect-information dual; Clontz and Holshouser previously showed these were Markov-information
dual in the case that \(X=C_p(Y)\)  \cite{2018arXiv180606001C}.

In the case that \(\mc B=D_X\), then \(\schStrongSelGame{\mc D_X}{\mc D_X}\) is the strong selective
separability game introduced in \cite{MR1711901}, and \(\schStrongSelGame{\mc T_X}{\neg\mc D_X}\) is the
point-picking game of Berner and Juh\'asz defined in \cite{MR775687}. Scheepers showed that these
were perfect-information dual in his paper.

\begin{proposition}
  \(\mc T_{X,x}\) is a reflection of \(\Omega_{X,x}\).
\end{proposition}
\begin{proof}
  For every set \(Y\) with limit point \(x\), the corresponding choice function \(f\in\mathbf C(\mc T_{X,x})\) is simply
  the witness that \(f(U)\in U\cap Y\).
\end{proof}

\begin{corollary}
  \(\schStrongSelGame{\Omega_{X,x}}{\mc B}\) and \(\schStrongSelGame{\mc T_{X,x}}{\neg\mc B}\) are perfect-information
  and Markov-information dual.
\end{corollary}

In the case that \(\mc B=\Gamma_{X,x}\) for some \(x\in X\), \(\schStrongSelGame{\mc T_{X,x}}{\neg\Gamma_{X,x}}\)
is Gruenhage's \(W\) game \cite{MR0413049}. Its dual \(\schStrongSelGame{\Omega_{X,x}}{\Gamma_{X,x}}\) characterizes
the strong Fr\'echet-Urysohn property \(\plI\notprewin\schStrongSelGame{\Omega_{X,x}}{\Gamma_{X,x}}\) at \(x\),
which now seen to be equivalent to \(\plII\notmarkwin\schStrongSelGame{\mc T_{X,x}}{\neg\Gamma_{X,x}}\).
This allows us to obtain the following result.

\begin{corollary}
  \(\plI\notprewin\schStrongSelGame{\Omega_{X,x}}{\Gamma_{X,x}}\)
  if and only if
  \(\plI\notwin\schStrongSelGame{\Omega_{X,x}}{\Gamma_{X,x}}\).
\end{corollary} 
\begin{proof}
  As shown in \cite{MR510910},
  a space is \(w\) at \(x\), that is, \(\plII\notwin\schStrongSelGame{\mc T_{X,x}}{\neg\Gamma_{X,x}}\) 
  if and only if \(\plI\notprewin\schStrongSelGame{\Omega_{X,x}}{\Gamma_{X,x}}\) for all \(x\in X\).
\end{proof}

For \(\mc B=\Omega_{X,x}\), \(\schStrongSelGame{\mc T_{X,x}}{\neg\Omega_{X,x}}\) is the variant of
Gruenhage's \(W\) game for clustering. This game is now seen to be dual to the strong countable fan tightness game
\(\schStrongSelGame{\Omega_{X,x}}{\Omega_{X,x}}\) at \(x\).

%
%

\section{Open Questions}

\begin{question}
  Does there exist a natural reflection for \(\Gamma_{X,x}\)
  or \(\Gamma_X=\{\mc U\subseteq\mc T_X:\forall x\in X(\mc U\setminus\mc T_{X,x}\in[T_X]^{<\aleph_0})\}\)?
\end{question} 

\begin{question}
  Can these results be extended for \(\schSelGame{\mc A}{\mc B}\)?
\end{question}

\section{Acknowledgements}

Thanks to Prof. Jared Holshouser for his input during the writing of these results.

\bibliographystyle{plain}
\bibliography{bibliography}

\end{document}